\documentclass[longtitle]{elsarticle}
\usepackage{amssymb,amsthm,amsmath}

\newtheorem{lem}{Lemma}
\newtheorem{thm}{Theorem}
\newtheorem{cor}{Corollary}

\newcommand{\car}{\square}

\newcommand{\diam}{\operatorname{diam}}
\newcommand{\capt}{\operatorname{capt}}

\begin{document}

\begin{abstract}
We consider the game of Cops and Robber played on 
the Cartesian product of two trees.
Assuming the players play perfectly,
it is shown that if there are two cops in the game,
then the length of the game 
(known as the 2-capture time of the graph) 
is equal to half the diameter of the graph.
In particular, 
the 2-capture time of the $m\times n$ grid
is proved to be $\lfloor \frac{m+n}{2}\rfloor -1$.
\end{abstract}


\author[co]{A. Mehrabian\corref{cor}}
\ead{amehrabian@uwaterloo.ca}
\cortext[cor]{Current address: Department of Combinatorics and Optimization, University of Waterloo, 200 University Avenue West, Waterloo, Ontario, Canada N2L 3G1}
\address[co]{Department of Computer Engineering\\
Sharif University of Technology\\
Azadi Avenue\\
Tehran, Iran}

\title{The capture time of grids}

\maketitle

\section{Introduction}

All graphs we consider are undirected, simple, finite, and connected.
Cops and Robber is a vertex pursuit game played on a graph, first introduced in  \cite{NW,Qu}.
We consider two teams of players, a set of $k$ cops $\mathcal C$ and a single robber $\mathcal R$. We think of the players (cops, robber) as occupying vertices, or being at vertices; a vertex can be occupied by more than one player. 
The two teams play in rounds.
In round 0, first each of the cops chooses a vertex to start, and then the robber chooses a vertex. In subsequent rounds, first each cop moves and then the robber moves. 
Each move takes one time unit and each player can either stay or go to an adjacent vertex in its move.
The cops win and the game ends if the robber is at a vertex occupied by some cop; otherwise, the robber wins.
Each player knows the other players' locations during the game.

As placing a cop on each vertex guarantees that the cops win, we may define the \emph{cop number} of a graph $G$, denoted by $c(G)$, as the minimum number of cops that have a winning strategy on $G$.
The graphs with cop number at most $k$ are \emph{$k$-cop-win} graphs.
For a survey of results on the cop number and related search parameters for graphs, see \cite{Alspach,Hahn}.

In a $k$-cop-win graph, how many moves does it take for the $k$ cops to win? To be more precise, say that the length of a game is $t$ if
the robber is captured in round $t$. The length is infinite if the robber can evade capture forever. We say that a play of the game
with $k$ cops is optimal if its length is the minimum over all possible games, assuming the robber is trying to evade capture
for as long as possible. 
There may be many optimal plays possible (for example, in a path with four vertices, one cop may start on either vertex of the centre), but the length of an
optimal game is an invariant of $G$. We denote this invariant by $\capt_k(G)$, and call it the \emph{$k$-capture time} of $G$.  The capture time of a graph may be viewed as the temporal counterpart of its cop number. 
The concept was first introduced in \cite{capturetime}, and is in part motivated by the fact that in real-world networks with limited resources, not only the number of cops but also the time it takes to capture the robber on the network is of practical importance.

It is known that for a fixed $k$, one can compute $\capt_k(G)$ (or decide that $G$ is not $k$-cop-win) in polynomial time \cite{kcopslrobbers}.
The capture time of 1-cop-win graphs was studied in \cite{capturetime,gavenciak}.
In this paper we study the 2-capture time of the Cartesian products of two trees, which are known to have cop number 2.
We prove that if $G$ is the Cartesian product of two trees, then $\capt_2(G) = \left\lfloor \diam(G) / 2\right\rfloor.$
In particular, the 2-capture time of the $m\times n$ grid is $\lfloor \frac{m+n}{2}\rfloor -1$.
It turns out that the techniques used here are not strong enough to determine the exact capture time of products of more than two trees (or higher-dimensional grids).
Hence the generalization of the results to more than two trees is left open (see section~\ref{sec_open} for some bounds).

In the following we will assume that after round 0, in all  subsequent rounds, the robber moves before the cops. This change clearly does not affect the cop number, and neither does it affect the capture time, because of the following observations: First, if a game is optimal, then the robber always gets caught on the cops' move. And second, since the robber is the player who decides her starting location last, making the first move does not give her any advantage.

Basic definitions come in section~\ref{sec_def}, and the main result is proved in section~\ref{sec_grid}. In section~\ref{sec_open} some open problems are raised.

\section{Definitions}
\label{sec_def}
Let $G$ be a graph.
We denote the distance between two vertices $u$ and $v$ in $G$ by $d(u,v)$.
The \emph{diameter} of $G$ is the maximum distance between any two vertices of $G$, and is denoted by $\diam(G)$.
For example, the diameter of a tree is the length of its longest paths.
We will assume that $\diam(T)>0$ for all trees $T$ in the following.
%The \emph{centre} of $G$ is a vertex $c$ such that $\max\{d(c,v):v\in V(G)\}$ is minimum.
The set of neighbors of a vertex $u$ is denoted by $N(u)$, and $N[u]$ is the union $N(u) \cup \{u\}$.
If $c(G) \leq k$ then we call an element $(c_1,c_2,\dots,c_k) \in \left(V(G)\right)^k$ a \emph{central $k$-tuple} if $k$ cops have an optimal strategy 
starting from these vertices.

\begin{thm}\label{onetree}
For any tree $T$, we have $\capt_1(T) = \lceil \diam(T) / 2\rceil$.
\end{thm}
\begin{proof}
Let $d= \diam(T)$, and let $t=\capt_1(T)$.
Let $P$ with vertices $a_1,\dots,a_{d+1}$ be a longest path in $T$, and let $v$ be a central $1$-tuple. 
We must have $d(v,a_1) \leq t$ since a robber that starts at $a_1$ and stays there forever should be captured in $t$ rounds, and similarly $d(v,a_{d+1}) \leq t$. 
This shows $d \leq 2t$ thus $t \geq \lceil d/2 \rceil$.

On the other hand, let $u= a_{1+\lceil d/2 \rceil}$; we claim that a cop starting from $u$ can capture the robber in at most $\lceil d/2 \rceil$ rounds.
The distance from any vertex to $u$ is at most $\lceil d/2 \rceil$. In every round, the cop moves towards the robber, and his distance from $u$ increases by $1$. Hence after at most $\lceil d/2 \rceil$ rounds the robber is captured.
\end{proof}

The \emph{Cartesian product} of two graphs $G$ and $H$ is a graph with vertex set $V(G) \times V(H)$, and with vertices $(u_1,v_1)$ and $(u_2,v_2)$ adjacent if either $u_1=u_2$ and $v_1v_2 \in E(H)$, or $v_1=v_2$ and $u_1u_2\in E(G)$.
We denote the Cartesian product of $G$ and $H$ by $G \car H$.
It is not hard to check that $\diam(G\car H) = \diam(G) + \diam(H)$.

\section{The capture time of the Cartesian product of trees}
\label{sec_grid}
The aim of this section is to prove that for trees $T_1$ and $T_2$,
$$\capt_2(T_1\car T_2) = \left\lfloor \frac{\diam(T_1\car T_2)}{2}\right\rfloor.$$
Note that it is known that $c(T_1\car T_2) = 2$ \cite{products}.

\subsection{Upper bound}

Let $T$ be a tree rooted at $s$. By \emph{going up} or \emph{down} we mean moving towards $s$ or away from $s$, respectively.
We say that vertex $v$ is a \emph{descendant} of vertex $u$ if $u$ is contained in the unique ($s,v$)-path (each vertex is a descendant of itself). The \emph{height} of a vertex $v$, written $h(v)$, is defined by $$h(v) = \max\{d(v, l): l\mathrm{\ is\ a\ leaf\ and\ a\ descendant\ of\ }v\}.$$

\begin{lem}
\label{lem_difficult}
Let $T_1, T_2$ be rooted trees.
We denote by $(u_1, u_2), (v_1, v_2)$, and $(r_1, r_2)$ the coordinates of $\mathcal{C}_1, \mathcal{C}_2$ (the first and second cops) and $\mathcal R$ (the robber) in $T_1 \car T_2$, respectively.
Suppose that the following conditions hold:
\begin{itemize}
\item $r_i$ is a descendant of $u_i$ and $v_i$ for $i=1,2$, 
\item $d(v_1,r_1) = 1 + d(u_1,r_1)$, 
\item $u_2 = v_2$, 
\item $d(u_2, r_2) \in \{d(u_1, r_1), d(v_1, r_1)\}$. 
\end{itemize}
If it is $\mathcal R$ to move, then the cops can capture her in at most $m$ rounds, where $m = \left\lfloor\left(h(u_1)+h(v_1)+h(u_2)+h(v_2)\right)/2\right\rfloor$.
\end{lem}
\begin{proof}
We use induction on $m$. If $m=0$, then $h(u_2)=h(v_2)=0$, since $u_2=v_2$. 
Thus $u_2$ is a leaf of $T_2$. The vertex $r_2$ is a descendent of $u_2$, so $r_2=u_2$.
Hence one of $d(u_1, r_1)$, $d(v_1, r_1)$ is also zero, and $\mathcal R$ is already captured.

Suppose that $m>0$. The robber either moves on $T_1$ (up or down), moves on $T_2$ (up or down), or does not move at all. Let $(r'_1, r'_2)$ be her position after the move. There are five cases to consider, in all of which we give a move for the cops, which results in an immediate capture or a similar situation with $m$ decreased by 1:
\begin{description}
\item[The robber moves up on $T_1$:] If $r_1=u_1$, then since $\mathcal R$ was not previously captured, 
$d(u_2,r_2) > 0 = d(u_1, r_1)$.
Since $d(u_2, r_2) \in \{d(u_1, r_1), d(v_1, r_1)\}$, we must have $d(u_2, r_2) = d(v_1, r_1)$.
On the other hand, $v_2=u_2$ and $d(v_1,r_1) = 1 + d(u_1,r_1)$.
Thus $d(v_2,r_2)=1$.
After $\mathcal R$'s move we have $r'_1=v_1$ since $v_1$ is the parent of $u_1$. Now $\mathcal{C}_2$ can immediately capture $\mathcal R$.
If $r_1 \neq u_1$, then $r'_1$ is still a descendent of $u_1$. The cops move down towards $r_2$ on $T_2$.
\item[The robber moves down on $T_1$:] The cops move down towards $r_1$ on $T_1$.
\item[The robber moves up on $T_2$:] 
If $r'_2=u_2$, then $\mathcal{C}_1$ can immediately capture the robber.
Otherwise, the cops move down towards $r_1$ on $T_1$. 
\item[The robber moves down on $T_2$:] The cops move down towards $r_2$ on $T_2$.
\item[The robber does not move:] If $d(u_2, r_2) = d(v_1, r_1)$, then the cops move down towards $r_2$ on $T_2$.
Otherwise, we have $d(u_2, r_2) = d(u_1, r_1)$, and the cops move down towards $r_1$ on $T_1$.
\end{description}
\end{proof}

\noindent\textbf{Remark.}
The point of $d(u_2, r_2) =d(v_2, r_2) \in \{d(u_1, r_1), d(v_1, r_1)\}$ is that $\mathcal R$ cannot pass through the cops in $T_1$ or $T_2$, since she will be immediately captured if she wants to go higher than $u_1$ (in $T_1$) or $u_2$ (in $T_2$).

\begin{lem}
If $T_1$ and $T_2$ are trees with diameter $2m+1$ and $2n$, respectively, then $\capt_2(T_1\car T_2) \leq m + n$.
\end{lem}
\begin{proof}
Let $P_1$ and $P_2$ be longest paths in $T_1$ and $T_2$, respectively. 
Label the vertices of $P_1$ as $a_1,a_2,\dots, a_{2m+2}$ and the vertices of $P_2$ as $b_1,b_2,\dots, b_{2n+1}$.
The cops start at $(u_1,u_2) = (a_{m+1},b_{n+1}), (v_1,v_2) = (a_{m+2},b_{n+1})$, and suppose that $\mathcal R$ goes to $(r_1,r_2)$ in her first move. We may assume, by symmetry, that $d(r_1, u_1) < d(r_1, v_1)$. Select $v_1, v_2$ as roots of $T_1, T_2$. Thus $h(u_1) = m$, $h(v_1) = m+1$, and $h(u_2)=h(v_2) = n$. If $d(u_2,r_2) \in \{d(u_1, r_1),d(v_1,r_1)\}$, then we can apply Lemma~\ref{lem_difficult}.
Otherwise, in the first phase of their strategy, the cops try to satisfy the conditions of Lemma~\ref{lem_difficult}.
If $d(u_2,r_2) < d(u_1, r_1)$, then the cops move down on $T_1$ towards $r_1$ to reduce $d(u_1, r_1)- d(u_2,r_2)$ to zero. If $d(u_2, r_2) > d(v_1, r_1)$, then the cops move down on $T_2$ towards $r_2$ to reduce $d(u_2, r_2) - d(v_1, r_1)$. In both cases they will eventually succeed to satisfy $d(u_2,r_2) \in \{d(u_1, r_1),d(v_1,r_1)\}$. As soon as they reach this situation, the first phase finishes. In the second phase, they will use the strategy given in Lemma~\ref{lem_difficult} to capture $\mathcal R$. We will now analyze the capture time of this strategy.

Initially $h(u_1) + h(v_1) + h(u_2) + h(v_2) = 2m+2n+1$. In each round of the first phase, the cops decrease the left side by 2. If the first phase takes $t$ rounds, then by Lemma~\ref{lem_difficult}, the second phase will take at most $\lfloor \left(2m+2n+1-2t\right)/2 \rfloor$ rounds. Consequently, $$\capt_2(T_1\car T_2) \leq t + \left \lfloor (2m+2n+1-2t)/2 \right\rfloor = m + n.$$
\end{proof}

\begin{cor}
When $T_1$ and $T_2$ are trees, $\capt_2(T_1\car T_2) \leq \left \lfloor \diam(T_1\car T_2)/2\right\rfloor$.
\end{cor}
\begin{proof}
Clearly $\diam(T_1\car T_2) = \diam(T_1) + \diam(T_2)$ and the corollary is proved by conditioning on the parity of $\diam(T_1)$ and $\diam(T_2)$, and noticing that adding a leaf to either of $T_1,T_2$ does not decrease the capture time of $T_1\car T_2$.
\end{proof}

\subsection{Lower bound}
A vertex $u$ is a \emph{corner} if there exists another vertex $v$ with $N[u] \subseteq N[v]$.

\begin{lem}\label{diamlowerbound}
Let $G$ be a $2$-cop-win graph, $u$ be a vertex of $G$ contained in an induced cycle $C$ of length $4$ with the property that for every vertex $v \in V(G)$, $|N(v) \cap V(C)| \leq 2$ holds. If $(c_1,c_2)$ is a central $2$-tuple, then $$d(u,c_1)+d(u,c_2) \leq 2\capt_2(G)+1.$$
\end{lem}
\begin{proof}
Let $t = \capt_2(G)$.
It can be seen that none of the vertices of $C$ is a corner, hence one cop alone cannot capture $\mathcal R$ even if $\mathcal R$ is restricted to move on $C$.
Let $d_1= d(c_1, u)$ and $d_2 = d(c_2, u)$, and assume, by symmetry, that $d_1 \leq d_2$. 
Since $\mathcal R$ will be captured in at most $t$ rounds if she starts at $u$ and remains there, we have $d_1 \leq t$. There are two cases to consider:
\begin{description}
\item[$d_1 = t$:] 
Suppose that $\mathcal R$ starts at $u$ and does not move in the first $t-1$ rounds.
Then after $t-1$ rounds, $\mathcal{C}_1$ should be in a vertex $c_1$ adjacent to $u$.
Since $u$ is not a corner, it has a neighbor $u'$ with $u' \notin N[c_1]$. 
Thus if $\mathcal R$ moves to $u'$ in round $t$, then $\mathcal{C}_2$ must immediately capture her.
Therefore $d_2 \leq t+1$. This gives $d_1+d_2 \leq 2t+1$.
\item[$d_1 < t$:] Since $\mathcal{C}_1$ alone cannot capture $\mathcal R$ even if $\mathcal R$ is restricted to move on $C$, $\mathcal{C}_2$ should be able to reach $C$ in at most $t$ rounds.
Therefore, $d_2 \leq t+2$ and $d_1 + d_2 \leq t-1 + t+2 = 2t+1$.
\end{description}
\end{proof}

\begin{lem}\label{lem_lower}
When $T_1$ and $T_2$ are trees, $\capt_2(T_1\car T_2) \geq \lfloor \diam(T_1\car T_2)/2\rfloor$.
\end{lem}
\begin{proof}
Let $G= T_1 \car T_2$, $m = \diam(T_1)$, and $n = \diam(T_2)$.
Let $P_1$, $P_2$ be longest paths in $T_1$ and $T_2$, respectively.
Label the vertices of $P_1$ as $a_1,a_2,\dots,a_{m+1}$, and vertices of $P_2$ as $b_1,b_2,\dots,b_{n+1}$.
Let $u = (a_1, b_1)$ and $v =(a_{m+1}, b_{n+1})$.
The vertex $u$ is contained in the cycle formed by the four vertices $(a_1,b_1)$, $(a_1,b_2)$, $(a_2,b_2)$, and $(a_2,b_1)$, and satisfies the conditions of Lemma~\ref{diamlowerbound}. 
The vertex $v$ satisfies the conditions of Lemma~\ref{diamlowerbound}, too.
Let $(c_1, c_2)$ be a central 2-tuple.
By triangle inequality and Lemma~\ref{diamlowerbound}, $$2\diam(G) = 2d(u,v) \leq d(c_1,u) + d(c_1,v) + d(c_2,u) + d(c_2,v) \leq 4\capt_2(G)+2.$$ We are done after dividing by 2 and noticing that $\capt_2(G)$ is an integer.
\end{proof}

\begin{thm}
For every two trees $T_1, T_2$ we have $$\capt_2(T_1\car T_2) = \left\lfloor \diam(T_1\car T_2) / 2\right\rfloor.$$
\end{thm}

\begin{cor}\label{T1T2carsum}
For every two trees $T_1, T_2$ we have $$\capt_1(T_1)+capt_1(T_2)-1 \leq \capt_2(T_1\car T_2) \leq \capt_1(T_1)+\capt_1(T_2).$$
\end{cor}

\begin{cor}
The $2$-capture time of the $m\times n$ grid is equal to $\lfloor (m+n)/2\rfloor - 1$.
\end{cor}

\begin{proof}
The $m\times n$ grid is the Cartesian product of a path with $m$ vertices and a path with $n$ vertices, and has diameter $m+n-2$.
\end{proof}

\section{Open problems}
\label{sec_open}
In the following we write $\capt(G)$ instead of $\capt_{c(G)}G$.
The cop number of the Cartesian product of $n$ trees is known to be $\lceil (n+1)/2 \rceil$, see \cite{products}.

We have seen that if $T_1,T_2$ are trees, then
$$\capt(T_1\car T_2) = \left\lfloor \frac12 \sum_{i=1}^{2} \diam(T_i)\right\rfloor.$$
It would be interesting to extend this result for more trees. 
Unfortunately, there is no easy generalization of the efficient capturing strategy given in Lemmas~1,~2.
Nevertheless, the author was able to prove the following bounds.

If $T_1,T_2,T_3$ are trees, then 
$$\left\lfloor \frac12 \sum_{i=1}^{3} \diam(T_i) \right\rfloor \leq \capt(T_1\car T_2 \car T_3) \leq 1 + \sum_{i=1}^{3} \diam(T_i).$$
The lower bound has a proof similar to Lemma~\ref{lem_lower}, and the proof for the upper bound uses some ideas from \cite{products} and some  new ideas.

Weaker bounds can be proved when the number of trees becomes larger than three:
If $T_1,T_2,\dots,T_n$ are trees, then $$\capt \left(\car_{i=1}^n T_i\right) \leq \sum_{i=1}^n \left(2^{\lceil \frac{i}{2} \rceil}-1\right) \diam(T_i).$$
Finding better bounds for higher-dimensional grids (which are Cartesian products of more than two paths) would also be interesting.

{}

\end{document}